\documentclass[12pt]{article}
\usepackage{graphicx}
\usepackage[a4paper, total={7in, 8in}]{geometry}
\usepackage{amssymb,amsmath,amsthm,amsfonts}
\usepackage{enumitem}
\usepackage{indentfirst}
\usepackage{float}
\usepackage{listings}
\usepackage{matlab-prettifier}
\usepackage{subfig}
\usepackage{bm}
\newcommand{\att}{\mathcal{A}_\mathcal{S}}
\newcommand{\bbr}[1]{\bm{(}#1\bm{)}}

\newcommand{\lip}{\text{Lip}}
\theoremstyle{definition}
\newtheorem{definition}{Definition}[section]
\theoremstyle{plain}
\newtheorem{theorem}{Theorem}[section]
\theoremstyle{definition}
\newtheorem{remark}{Remark}[section]
\theoremstyle{definition}
\newtheorem{corollary}{Corollary}[section]
\theoremstyle{definition}

\theoremstyle{plain}
\newtheorem{proposition}{Proposition}[section]
\numberwithin{equation}{section}
\usepackage{aligned-overset}
\newtheorem{lemma}{Lemma}[section]

\begin{document}
\author{Bogdan Anghelina,
  Radu Miculescu, María Antonia Navascués
}
\date{}
\title{On the range of fractal interpolation functions}
\maketitle

\textbf{Abstract.} In this paper, based on the results from  [On the localization of Hutchinson–Barnsley fractals, Chaos Solitons Fractals, 173 (2023), 113674], we generate coverings
(consisting of finite families of rhombi) of the graph of fractal interpolation
functions. As a by-product we obtain estimations for the range of such
functions. Some concrete examples and graphical representations are
provided.
\\

\\

\section{Introduction}
Fractal interpolation functions (abbreviated FIFs) are extremely useful instruments for
modeling irregular patterns which produce smooth and non-smooth
interpolants. They have been introduced by M. Barnsley, via the
concept of iterated function system (see \cite{hutchinson}), including and supplementing the
classical interpolants.
Following the seminal work of M. Barnsley (see \cite{barnsley1}), the fractal
interpolation functions theory has been expanded in several ways. For
example:
\\- affine zipper fractal interpolation functions are considered in
\cite{chand}
\\ - vector valued fractal interpolation functions are treated in
\cite{massopust1}
\\- local fractal interpolation functions are examined in \cite{massopust3}
\\- fractal interpolation functions with variable parameters are
investigated in \cite{wang}
\\- graph-directed fractal interpolation functions are studied in
\cite{deniz}
\\- fractal interpolation functions with partial self similarity are
explored in \cite{leor}.\\
See also \cite{miculescu1}, \cite{navasques1}, \cite{navasques6}, \cite{navasques5}, \cite{ri} and \cite{vasilev}.

Some fields where one can find applications of FIFs encompass:
\\- signal analysis (see \cite{banerje})
\\- satellite images (see \cite{cheng})
\\- medicine (see \cite{craciunescu} and \cite{stanley})
\\- image compression (see \cite{drako})
\\- financial analysis (see \cite{leon})
\\- ecology (see \cite{lu})
\\- signal reconstruction (see \cite{Gdaw}, \cite{navasques2} and \cite{zhai})
\\- geology (see \cite{xie})
\\- meteorology (see \cite{xiu}).

Excellent surveys on fractal interpolation functions could be found in \cite{banerje}, \cite{barnsley2}, \cite{massopust2}, \cite{navasques3} and \cite{navasques4}.

The brusque oscillations occurring in fractal interpolation
functions modeling financial data (which are of great interest for
decision makers and investors see \cite{manous}) is a strong reason to study
the range of a FIF. An additional motive, from the point of view of
cardiologists, for such a study stems from \cite{riccio}. Actually such an
investigation, but in some particular cases, has been done in \cite{chen} and
\cite{viswanathan}.

In this paper, based on the results from \cite{locHBfrct}, we generate coverings
(consisting of finite families of rhombi) of the graph of fractal interpolation
functions (see Theorem 3.1). As a by-product we obtain estimations for
the range of such functions (see Theorem 3.2). Some concrete examples
and graphical representations are provided in the last section.
\section{Preliminaries}

Given a metric space \((X,d)\), \(x\in X\) and \(r>0\) we shall use the following notation:

\begin{itemize}
  \item \(B[x,r]=\{y\in X \;|\; d(y,x)\leq r\}\)

  \item \(\mathcal{P}_{cp}(X)=\{A\subseteq X| A \text{ is non-empty and compact}\}\)
  \item \(h\) is the Hausdorff-Pompeiu metric, described by
        \[\displaystyle h(K_1,K_2)=\max\left\{\sup_{x\in K_1}d(x,K_2),\sup_{x\in K_2} d(x,K_1)\right\},\]
        for all \(K_1,K_2\in\mathcal{P}_{cp}(X)\).

\end{itemize}

For a Lipschitz function \(f:X \to X\) we shall denote by \(\lip(f)\) the Lipschitz constant of \(f\).
\\

Let us recall the following well known result:
\begin{lemma}\label{unionsup}
  Let \(A_1,\;A_2,\dots,\; A_n\), where \(n\in\mathbb{N}\), be non-empty, bounded subsets of \(\mathbb{R}\). Then

  \[\sup(A_1\cup A_2\cup \dots\cup A_n ) = \max \{\sup A_1, \sup A_2, \dots, \sup A_n\}.\]
\end{lemma}
\begin{definition}[]\label{fractaldef} An iterated function system (for short IFS) is a pair \(\mathcal{S}=((X,d),(f_i)_{i\in I})\), where \((X,d)\) is a complete metric space and \(f_i:X\to X\), \(i\in I\), are Banach contractions.

  The function \(F_\mathcal{S}:\mathcal{P}_{cp}(X)\to \mathcal{P}_{cp}(X)\), given by \[F_\mathcal{S}(K)=\bigcup_{i\in I}f_i(K),\] for every \(K\in \mathcal{P}_{cp}(X)\), is called the fractal operator (or the Hutchinson operator) associated with \(\mathcal{S}\).

\end{definition}
\begin{proposition}\label{2:picard} If \(\mathcal{S}=((X,d),(f_i)_{i\in I})\) is an iterated function system, then \(F_\mathcal{S}\) is a contraction with respect to \(h\). Its unique fixed point is denoted by \(\att\) and it is called the attractor of \(\mathcal{S}\) since
  \[\lim_{n\to\infty} \underbrace{(F_\mathcal{S}\circ F_\mathcal{S}\circ\dots\circ F_\mathcal{S})}_{n \text{ times }}(K) = \att,\]
  for each \(K\in\mathcal{P}_{cp}(X).\)
\end{proposition}
\begin{theorem}[see Theorem 3.1 from \cite{locHBfrct}]
  For an iterated function system \(\mathcal{S}=((X,d),(f_i)_{i\in \{1,\dots n\}})\), we have
  \[\att\subseteq \left(\bigcup_{i\in\{1,\dots,\;n-1\}}B\left[\gamma_{i},Ms_{i}\frac{1+s_{n}}{1-s_{{n-1}}s_{n}}\right]\right)\bigcup B\left[\gamma_{n},Ms_{n}\frac{1+s_{{n-1}}}{1-s_{{n-1}}s_{n}}\right],\]

  where \begin{enumerate}
    \item \(s_{i} = \lip(f_{i})\) for each \(i\in\{1,\dots ,n\}\)
          \item\(s_{1}\leq s_{2}\leq \dots \leq s_{n}\)
    \item \(\gamma_{i}\) is the unique fixed point of \(f_{i}\) for each \(i\in\{1,\dots ,n\}\)
    \item \(M=\max_{i,j\in\{1,\dots,n\}}d(\gamma_i,\gamma_j)\).
  \end{enumerate}

\end{theorem}
\section{The main results}

\textbf{The framework of our main result}
\\

Inspired by \cite{barnsley1} and \cite{barnsley2}, let \(n\in\mathbb{N}\) and \((x_0,y_0),(x_1,y_1),\dots, (x_n,y_n)\in\mathbb{R}^2\) with \[a:=x_0<x_1<\dots<x_n=:b\]
and let us consider the iterated function system \(\mathcal{S}=((\mathbb{R}^2,\rho),(f_k)_{k\in \{1,\dots, n\}})\), where:
\begin{itemize}
  \item
        \(f_k:\mathbb{R}^2\to\mathbb{R}\) is given by \[f_k(x,y)=(a_k x+ b_k , c_k x +d_k y+e_k),\] for all \((x,y)\in\mathbb{R}^2\) and \(k\in\{1,\dots,n\}\), where
        \[\begin{aligned}
            a_k & =\frac{x_k-x_{k-1}}{x_n-x_0}                                             \\
            b_k & = \frac{x_nx_{k-1}-x_0x_k}{x_n-x_0}                                      \\
            d_k & \in[0,1)                                                                 \\
            c_k & = \frac{y_k-y_{k-1}}{x_n-x_0}-d_k\frac{y_n-y_0}{x_n-x_0}                 \\
            e_k & =\frac{x_ny_{k-1}-x_0y_k}{x_n-x_0} -d_k \frac{x_ny_{0}-x_0y_n}{x_n-x_0};
          \end{aligned}\]

  \item the complete metric \(\rho:\mathbb{R}^2\times\mathbb{R}^2\to[0,\infty)\) is given by
        \[\rho((u_1,v_1),(u_2,v_2))=|u_1-u_2|+\theta |v_1-v_2|,\] for all \((u_1,v_1),(u_2,v_2) \in \mathbb{R}^2\), where
        \[\theta = \left\{\begin{aligned}
             & 1,                                                                   &  & c_1=c_2=\dots=c_n=0; \\
             & \frac{1-\max_{k\in\{1\dots n\}}a_k}{2\max_{k\in\{1,\dots,n\}}|c_k|}, &  & \text{otherwise}.
          \end{aligned}\right.\]
\end{itemize}
Then, there exists a continuous function \(f:[a,b]\to\mathbb{R}\), which is called an affine fractal interpolation function, such that :

\begin{center}\begin{itemize}
    \item \[f(x_k)=y_k,\] for all \(k\in\{0,1,\dots, n\}\)
    \item \[G_f=\att.\]
  \end{itemize}\end{center}

\textbf{Finding the fixed point of \(f_k\)}
\\

For a fixed, but arbitrarily chosen \(k\in\{1,\dots,n\}\), from \(f_k(\gamma_k)=\gamma_k=(u_k,v_k)\) we obtain
\[(a_ku_k+b_k,c_ku_k+d_kv_k+e_k)=(u_k,v_k),\]
hence
\[\begin{cases}
    a_ku_k+b_k        & =u_k  \\
    c_ku_k+d_kv_k+e_k & =v_k.
  \end{cases}\]

From the first equation, we deduce that \[u_k = \frac{b_k}{1-a_k}.\]

Substituting in the second equation, we obtain
\[c_k\frac{b_k}{1-a_k}+d_kv_k+e_k=v_k,\]
so
\[v_k =\frac{b_kc_k}{(1-a_k)(1-d_k)}+\frac{e_k}{1-d_k}.\]

Hence, the (unique) fixed point of \(f_k\) is \[\gamma_k = \left(\frac{b_k}{1-a_k},\frac{b_kc_k}{(1-a_k)(1-d_k)}+\frac{e_k}{1-d_k}\right),\]
for all \(k\in\{1,\dots ,n\}\).
\\\\
\textbf{Finding the Lipschitz constant of \(f_k\)}
\\

\begin{proposition}
  In the above mentioned framework, we have
  \[\lip(f_k)=\max \{d_k,a_k+\theta |c_k|\},\] for every \(k\in\{1,\dots,n\}\)
\end{proposition}
\begin{proof}
  Let us consider \(k\in\{1,\dots, n\}\) arbitrarily chosen, but fixed.

  Note that \[\lip(f_k)= \sup \mathcal{M},\] where \[\begin{aligned}\mathcal{M}= & \left\{ \frac{\rho\bbr{f_k(u_1,v_1),f_k(u_2,v_2)}}{\rho\bbr{(u_1,v_1),(u_2,v_2)}} \big| {(u_1,v_1),(u_2,v_2)\in\mathbb{R}^2, (u_1,v_1)\neq (u_2,v_2)}\right\}        \\
               =            & \left\{\frac{a_k|u_1-u_2|+\theta|c_k(u_1-u_2)+d_k(v_1-v_2)|}{|u_1-u_2|+\theta |v_1-v_2|}\big| {(u_1,v_1),(u_2,v_2)\in\mathbb{R}^2, (u_1,v_1)\neq (u_2,v_2)}\right\}.
    \end{aligned}\]

  If \(d_k=0\), we get \[\begin{aligned}
      \mathcal{M}= & \left\{\frac{a_k|u_1-u_2|+\theta|c_k(u_1-u_2)|}{|u_1-u_2|+\theta |v_1-v_2|}\big| {(u_1,v_1),(u_2,v_2)\in\mathbb{R}^2,
      v_1\neq v_2, u_1=u_2}\right\}                                                                                                                                        \\
                   & \qquad\bigcup\left\{\frac{a_k|u_1-u_2|+\theta|c_k(u_1-u_2)|}{|u_1-u_2|+\theta |v_1-v_2|}\big| (u_1,v_1),(u_2,v_2)\in\mathbb{R}^2,u_1\neq u_2\right\}  \\
      =            & \{0\}\cup\left\{\frac{(a_k+\theta|c_k|)|u_1-u_2|}{|u_1-u_2|+\theta \left|v_1-v_2\right|}\big| (u_1,v_1),(u_2,v_2)\in\mathbb{R}^2, u_1\neq u_2\right\} \\
      =            & \{0\}\cup\left\{\frac{a_k+\theta|c_k|}{1+\theta \left|\frac{v_1-v_2}{u_1-u_2}\right|}\big| (u_1,v_1),(u_2,v_2)\in\mathbb{R}^2, u_1\neq u_2\right\}    \\
      =            & \{0\}\cup\left\{\frac{a_k+\theta|c_k|}{1+\theta \left|t\right|}\big|t\in\mathbb{R}\right\}.
    \end{aligned}\]

  Since \[\sup\left\{\frac{a_k+\theta|c_k|}{1+\theta \left|t\right|}\big|t\in\mathbb{R}\right\}= a_k+\theta |c_k|,\]\
  taking into account Lemma \ref{unionsup}, we get \[\lip(f_k)=\max\{0,a_k+\theta|c_k|\}=\max\{d_k,a_k+\theta c_k\}.\]

  So, we can suppose that \(d_k>0\).

  We have

  \[\begin{aligned}
      \mathcal{M}= & \left\{\frac{a_k|u_1-u_2|+\theta|c_k(u_1-u_2)+d_k(v_1-v_2)|}{|u_1-u_2|+\theta |v_1-v_2|}\big| {(u_1,v_1),(u_2,v_2)\in\mathbb{R}^2,
      v_1\neq v_2, u_1=u_2}\right\}                                                                                                                                                                  \\
                   & \qquad\bigcup\left\{\frac{a_k|u_1-u_2|+\theta|c_k(u_1-u_2)+d_k(v_1-v_2)|}{|u_1-u_2|+\theta |v_1-v_2|}\big| (u_1,v_1),(u_2,v_2)\in\mathbb{R}^2,u_1\neq u_2\right\}               \\
      =            & \{d_k\}\cup\left\{\frac{a_k|u_1-u_2|+\theta|c_k(u_1-u_2)+d_k(v_1-v_2)|}{|u_1-u_2|+\theta \left|v_1-v_2\right|}\big| (u_1,v_1),(u_2,v_2)\in\mathbb{R}^2, u_1\neq u_2\right\}     \\
      =            & \{d_k\}\cup\left\{\frac{a_k+\theta|c_k+d_k\frac{v_1-v_2}{u_1-u_2}|}{1+\theta \left|\frac{v_1-v_2}{u_1-u_2}\right|}\big| (u_1,v_1),(u_2,v_2)\in\mathbb{R}^2, u_1\neq u_2\right\} \\
      =            & \{d_k\}\cup\left\{\frac{a_k+\theta|c_k+d_kt|}{1+\theta \left|t\right|}\big|t\in\mathbb{R}\right\}.
    \end{aligned}\]

  We divide our proof in three cases, namely:
  \begin{enumerate}
    \item[] Case 1. \(c_k =0 \)
    \item[] Case 2. \(c_k < 0\)
    \item[] Case 3. \(c_k > 0\)
  \end{enumerate}

  In the first case, we have
  \[\mathcal{M}=\{d_k\}\cup\left\{\frac{a_k+\theta d_kx}{1+\theta x}\big|x \geq 0\right\}.\]

  Since
  \[\sup\left\{\frac{a_k+\theta d_kx}{1+\theta x}\big|x \geq 0\right\}=\begin{cases}
      d_k,                 & \qquad d_k\geq a_k \\
      a_k=a_k+\theta|c_k|, & \qquad d_k<a_k
    \end{cases}, \]
  via Lemma \ref{unionsup}, we get
  \[\lip(f_k)=\sup \mathcal{M} = \max \{d_k,a_k+\theta |c_k|\}.\]

  In the second case, we have
  \[\mathcal{M}=\{d_k\} \cup \mathcal{M}_1 \cup \mathcal{M}_2 \cup \mathcal{M}_3,\]
  where
  \begin{align*}
    \mathcal{M}_1 & = \left\{\frac{a_k+\theta c_k+\theta d_kt}{1+\theta t}\big|t\geq -\frac{c_k}{d_k}\right\}, \\
    \mathcal{M}_2 & = \left\{\frac{a_k-\theta c_k-\theta d_kt}{1+\theta t}\big|0<t< -\frac{c_k}{d_k}\right\}   \\
    \intertext{and}
    \mathcal{M}_3 & = \left\{\frac{a_k-\theta c_k-\theta d_kt}{1-\theta t}\big|t\leq0\right\}.
  \end{align*}

  Studying the monotonicity of the expressions that define the sets \(\mathcal{M}_1\), \(\mathcal{M}_2\) and \(\mathcal{M}_3\), we infer that
  \begin{align*}
    \sup \mathcal{M}_1 & =\begin{cases}
                            d_k,                           & \qquad a_k+\theta c_k\leq d_k \\
                            \frac{a_kd_k}{d_k-\theta c_k}, & \qquad a_k+\theta c_k>d_k
                          \end{cases},    \\
    \sup \mathcal{M}_2 & = a_k-\theta c_k=a_k+\theta|c_k|                                   \\\intertext{and}
    \sup \mathcal{M}_3 & =\begin{cases}
                            a_k-\theta c_k=a_k+\theta |c_k| , & \qquad a_k-\theta c_k> d_k    \\
                            d_k ,                             & \qquad a_k-\theta c_k\leq d_k
                          \end{cases}.\end{align*}

  Based on Lemma \ref{unionsup}, we get
  \[\lip(f_k) = \sup \mathcal{M} = \max \{d_k,\sup \mathcal{M}_1, \sup \mathcal{M}_2, \sup \mathcal{M}_3\},\] hence
  \begin{equation}\label{supge0}
    \lip(f_k)=\max\{d_k,a_k+\theta|c_k|,\sup \mathcal{M}_1,\sup \mathcal{M}_3\}.
  \end{equation}

  Further, we divide our discussion into three subcases, specifically:
  \begin{enumerate}
    \item[\((\alpha)\)] \(d_k< a_k+\theta c_k\)
    \item[\((\beta)\)] \(a_k+\theta c_k\leq d_k < a_k-\theta c_k\)
    \item[\((\gamma)\)] \(a_k-\theta c_k\leq d_k\).
  \end{enumerate}

  Note that in case \((\alpha)\)  we have \(d_k<  a_k+\theta c_k< a_k-\theta c_k\) and in case \((\gamma)\) we have \(a_k+\theta c_k<a_k-\theta c_k\leq d_k\).

  In case \((\alpha)\), \eqref{supge0} yields
  \[\lip(f_k) = \max \left\{d_k, a_k+\theta |c_k|, \frac{a_kd_k}{d_k-\theta c_k}\right\}=\max\{d_k, a_k+\theta |c_k|\},\]
  since \[ \frac{a_kd_k}{d_k-\theta c_k}\leq a_k+\theta |c_k|.\]

  In cases \((\beta)\) and \((\gamma)\), via \eqref{supge0}, we get
  \[\lip(f_k) = \max \left\{d_k, a_k+\theta |c_k|\right\}.\]

  In the third case, we have

  \[\mathcal{M}=\{d_k\} \cup \mathcal{N}_1 \cup \mathcal{N}_2 \cup \mathcal{N}_3,\]
  where
  \begin{align*}
    \mathcal{N}_1 & = \left\{\frac{a_k+\theta c_k+\theta d_kt}{1+\theta t}\big|t\geq 0\right\},               \\
    \mathcal{N}_2 & = \left\{\frac{a_k+\theta c_k+\theta d_kt}{1-\theta t}\big| -\frac{c_k}{d_k}<t<0\right\}  \\\intertext{and}
    \mathcal{N}_3 & = \left\{\frac{a_k-\theta c_k-\theta d_kt}{1-\theta t}\big|t\leq-\frac{c_k}{d_k}\right\}.
  \end{align*}

  Since
  \begin{align*}
    \sup \mathcal{N}_1 & =\begin{cases}
                            d_k,                            & \qquad a_k+\theta c_k\leq d_k \\
                            a_k+\theta c_k=a_k+\theta|c_k|, & \qquad a_k+\theta c_k>d_k
                          \end{cases}, \\
    \sup \mathcal{N}_2 & = a_k+\theta c_k=a_k+\theta|c_k|                                 \\
    \intertext{and}
    \sup \mathcal{N}_3 & =\begin{cases}
                            \frac{a_kd_k}{d_k+\theta c_k} , & \qquad a_k-\theta c_k> d_k    \\
                            d_k ,                           & \qquad a_k-\theta c_k\leq d_k
                          \end{cases},
  \end{align*}
  based on Lemma \ref{unionsup}, we obtain
  \[\lip(f_k) = \sup \mathcal{M} = \max \{d_k,\sup \mathcal{N}_1, \sup \mathcal{N}_2, \sup \mathcal{N}_3\},\] hence
  \begin{equation}\label{suple0}
    \lip(f_k)=\max\{d_k,a_k+\theta|c_k|,\sup \mathcal{N}_1,\sup \mathcal{N}_3\}.
  \end{equation}

  Further, we divide our discussion into three subcases, specifically:
  \begin{enumerate}
    \item[\((\alpha)\)] \(d_k< a_k-\theta c_k\)
    \item[\((\beta)\)] \(a_k-\theta c_k\leq d_k < a_k+\theta c_k\)
    \item[\((\gamma)\)] \(a_k+\theta c_k\leq d_k\).
  \end{enumerate}

  Note that in case \((\alpha)\) we have \(d_k<  a_k-\theta c_k< a_k+\theta c_k\) and in case \((\gamma)\) we have \(a_k-\theta c_k<a_k+\theta c_k\leq d_k\).

  In case \((\alpha)\), \eqref{suple0} ensures
  \[\lip(f_k) = \max \left\{d_k, a_k+\theta |c_k|, \frac{a_kd_k}{d_k+\theta c_k}\right\}=\max\{d_k, a_k+\theta |c_k|\},\]
  since \[ \frac{a_kd_k}{d_k+\theta c_k}\leq a_k+\theta |c_k|.\]

  In cases \((\beta)\) and \((\gamma)\), via \eqref{suple0}, we get
  \[\lip(f_k) = \max \left\{d_k, a_k+\theta |c_k|\right\}.\]

  All in all, we obtain
  \[\lip(f_k) = \max \left\{d_k, a_k+\theta |c_k|\right\}.\]
\end{proof}

\begin{remark}
  In view of the choice of \(d_k\) and \(\theta\), we have \[\lip(f_k)<1,\] for all \(k\in\{1,\dots, n\}\).
\end{remark}

\textbf{Describing the balls with respect to the metric \(\rho\)}
\\

For \(r\geq 0\) and \((x,y)\in\mathbb{R}^2\), the closed ball with center \((x,y)\) and radius \(r\), with respect to \(\rho\), is the set
\[\begin{aligned}\{ & (u,v)\in\mathbb{R}^2|\rho\bbr{(x,y),(u,v)}\leq r\}
                  & =\{(u,v)\in\mathbb{R}^2||x-u|+\theta|y-v|\leq r\}=:R[(x,y),r],
  \end{aligned}\]
whose border is the rhombus having the vertixes \(V_1\left(x+r,y\right)\), \(V_2\left(x-r,y\right)\), \(V_3\left(x,y+\frac{r}{\theta}\right)\) and  \(V_4\left(x,y-\frac{r}{\theta}\right).\)

\begin{remark}\label{span}
  Note that
  \[\max\{v\in\mathbb{R}| \text{ there exists } u\in \mathbb{R} \text{ such that } (u,v)\in R[(x,y),r]\}=y+\frac{r}{\theta}\]
  and
  \[\min\{v\in\mathbb{R}| \text{ there exists } u\in \mathbb{R} \text{ such that } (u,v)\in R[(x,y),r]\}=y-\frac{r}{\theta}.\]
\end{remark}

\textbf{Obtaining a cover for \(G_f\)}
\\

Summing up, in the previous framework, via Theorem 2.1, by relabeling the functions \(f_k\) such that \(s_1\leq s_2\leq\dots \leq s_n\), we get the following result:
\begin{theorem}\label{incl}
  \[G_f\subseteq \left(\bigcup_{k\in\{1,\dots,\;n-1\}}R\left[\gamma_{k},Ms_{k}\frac{1+s_{n}}{1-s_{{n-1}}s_{n}}\right]\right)\bigcup R\left[\gamma_{n},Ms_{n}\frac{1+s_{{n-1}}}{1-s_{{n-1}}s_{n}}\right],\] where
  \[\gamma_k = \left(\frac{b_k}{1-a_k},\frac{b_kc_k}{(1-a_k)(1-d_k)}+\frac{e_k}{1-d_k}\right),\]
  \[s_k = \max\{a_k+\theta|c_k|,d_k\},\]for all \(k\in\{1,\dots ,n\}\),
  \[\theta = \left\{\begin{aligned}
    & 1,                                                                   &  & c_1=c_2=\dots=c_n=0; \\
    & \frac{1-\max_{k\in\{1\dots n\}}a_k}{2\max_{k\in\{1,\dots,n\}}|c_k|}, &  & \text{otherwise}
 \end{aligned}\right.\]
  and
  \[M= \max_{i,j\in\{1,\dots,n\}}\left\{\left|\frac{b_i}{1-a_i}-\frac{b_j}{1-a_j}\right|+\theta \left|\frac{c_ib_i}{(1-a_i)(1-d_i)}+\frac{e_i}{1-d_i}-\frac{c_jb_j}{(1-a_j)(1-d_j)}-\frac{e_j}{1-d_j} \right|\right\}.\]
\end{theorem}
The set \[\mathcal{C_\mathcal{S}}:=\left(\bigcup_{k\in\{1,\dots,\;n-1\}}R\left[\gamma_{k},Ms_{k}\frac{1+s_{n}}{1-s_{{n-1}}s_{n}}\right]\right)\bigcup R\left[\gamma_{n},Ms_{n}\frac{1+s_{{n-1}}}{1-s_{{n-1}}s_{n}}\right]\] will be called the covering of \(G_f\) associated to the system \(\mathcal{S}\).
For \(m\in\mathbb{N}\), we will denote \[\mathcal{C}_{\mathcal{S}_m}=\mathcal{C}_m,\] where
\[\mathcal{S}_m=\left((\mathbb{R}^2,\rho), (f_{k_1}\circ\dots \circ f_{k_m})_{k_1,\dots,k_m\in\{1,\dots ,n\}}\right).\]

\begin{remark}
  According to Theorem 3.2 from \cite{locHBfrct}, by applying Theorem 3.1 for the IFSs \(\mathcal{S}_m\), \(m\in\mathbb{N}\), we obtain coverings of \(G_f\) with balls having the radii as small as we want.

  Let us note that, in view of Remark 3.2 from \cite{locHBfrct}, we have
  \[\lim_{m\to\infty} h(\mathcal{C}_m,G_f)=0.\]
\end{remark}

\begin{remark}
  In the above mentioned framework, for \(k_1, k_2, \dots k_m\in \{1,\dots,n\} \), with \(m\in \mathbb{N}\), \(m\geq 2\), we have
  \[(f_{k_1}\circ\dots \circ f_{k_m})(x,y)=(a_{k_1k_2\dots k_m}x+b_{k_1k_2\dots k_m}, c_{k_1k_2,\dots k_m}x+d_{k_1k_2,\dots k_m}y+e_{k_1k_2\dots k_m}),\]
  for all \(x,y\in\mathbb{R}\), where \(a_{k_1k_2\dots k_m}\), \(b_{k_1k_2\dots k_m}\), \(c_{k_1k_2\dots k_m}\) , \(d_{k_1k_2\dots k_m}\) and \(e_{k_1k_2\dots k_m}\) can be obtained recursively from the following formulas
  \[\left\{
    \begin{aligned}
      a_{k_1k_2\dots k_m} & = a_{k_1} a_{k_2 \dots k_m}                                     \\
      b_{k_1k_2\dots k_m} & = a_{k_1} b_{k_2 \dots k_m}+b_{k_1}                             \\
      c_{k_1k_2\dots k_m} & = c_{k_1}a_{k_2 \dots k_m}+d_{k_1} c_{k_2k_3 \dots k_m}         \\
      d_{k_1k_2\dots k_m} & = d_{k_1} d_{k_2 \dots k_m}                                     \\
      e_{k_1k_2\dots k_m} & = c_{k_1} b_{k_2 \dots k_m}+d_{k_1} e_{k_2 \dots k_m}+ e_{k_1},
    \end{aligned}\right.\] which represent a very good tool for improving the efficiency of the computations involved.

\end{remark}

\textbf{Obtaining an estimation for \(Imf\)}
\\

Based on Remark \ref{span} and Theorem \ref{incl}, we obtain the following result:
\begin{theorem}\label{infinclusion}
  In the previous framework, we have: \[Imf \subseteq [A,B],\] where
  \[\begin{aligned}
      A=\min & \left\{
      \min_{k\in\{1,\dots,n-1\}}\left(
      \frac{b_kc_k}{(1-a_k)(1-d_k)}+
      \frac{e_k}{1-d_k}-\frac{Ms_k(1+s_n)}{\theta(1-s_{n-1}s_n)}\right),
      \right.
      \\ &
      \left.\frac{b_nc_n}{(1-a_n)(1-d_n)}+\frac{e_n}{1-d_n}-\frac{Ms_n(1+s_{n-1})}{\theta(1-s_{n-1}s_n)} \right\}\end{aligned}\]
  and

  \[\begin{aligned}
      B=\max & \left\{
      \max_{k\in\{1,\dots,n-1\}}\left(
      \frac{b_kc_k}{(1-a_k)(1-d_k)}+
      \frac{e_k}{1-d_k}+\frac{Ms_k(1+s_n)}{\theta(1-s_{n-1}s_n)}\right),
      \right.
      \\ &
      \left.\frac{b_nc_n}{(1-a_n)(1-d_n)}+\frac{e_n}{1-d_n}+\frac{Ms_n(1+s_{n-1})}{\theta(1-s_{n-1}s_n)} \right\}.\end{aligned}\]
\end{theorem}

\begin{corollary} \textit{In the previous framework, we have}
  \[G_f\subseteq [a,b]\times[A,B].\]
\end{corollary}
\begin{remark}
  Note that, \(A\) and \(B\) can be algebraically expressed in terms of the data set \((x_0,y_0)\), \;\((x_1,y_1),\dots\), \((x_n,y_n)\) and the scaling factors \(d_k\).
\end{remark}
\section{Examples}
We are going to apply our results for the following frameworks (which are considered in \cite{manous2}, \cite{manous} and respectively \cite{Gdaw}):\\
\\

\textbf{Framework 1}

\begin{itemize} \item the data set is \begin{center} \(\{(0,3),(1,2),(2,4),(3,3),(4,4)\}\)\end{center}
  \item \;
        \begin{center}
          \(d_k=\frac{3}{10},\)
        \end{center}
\end{itemize}
for each \(k\in\{1,2,3,4\}\).

The pictures \(F_{11},F_{12},F_{13},F_{14}\) and \(F_{15}\) contain the graphical representations of \(\mathcal{C}_1\), \(\mathcal{C}_2\), \(\mathcal{C}_3\), \(\mathcal{C}_4\) and respectively \(\mathcal{C}_5\).
\captionsetup[subfigure]{labelformat=empty}
\begin{figure}[H]
  \centering
  \begin{tabular}{cc}
    \subfloat[\(F_{11}:\) Visualization of \(\mathcal{C}_1\)]{\includegraphics[width = 0.5\textwidth]{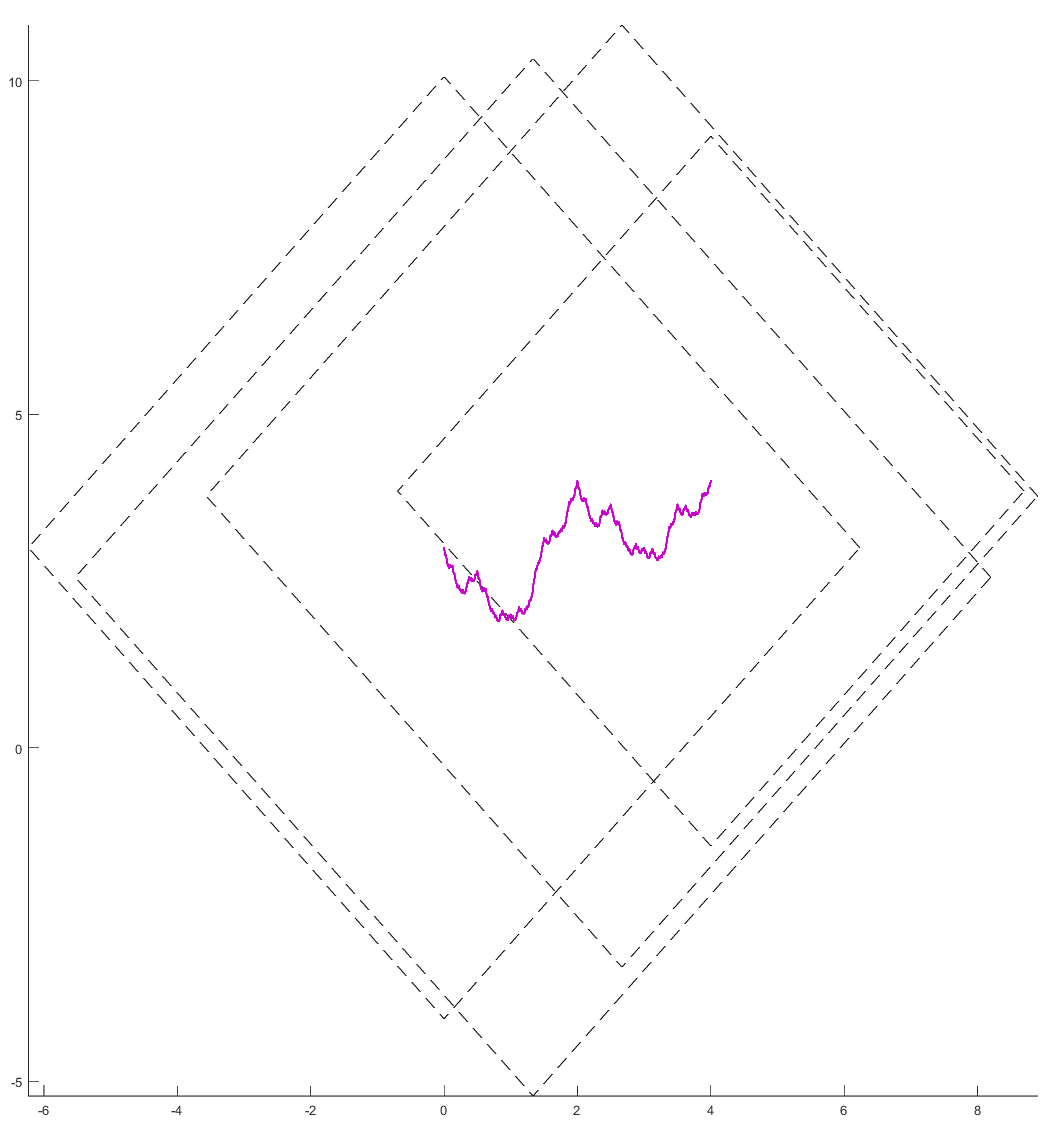}} &
    \subfloat[\(F_{12}:\) Visualization of \(\mathcal{C}_2\)]{\includegraphics[width = 0.5\textwidth]{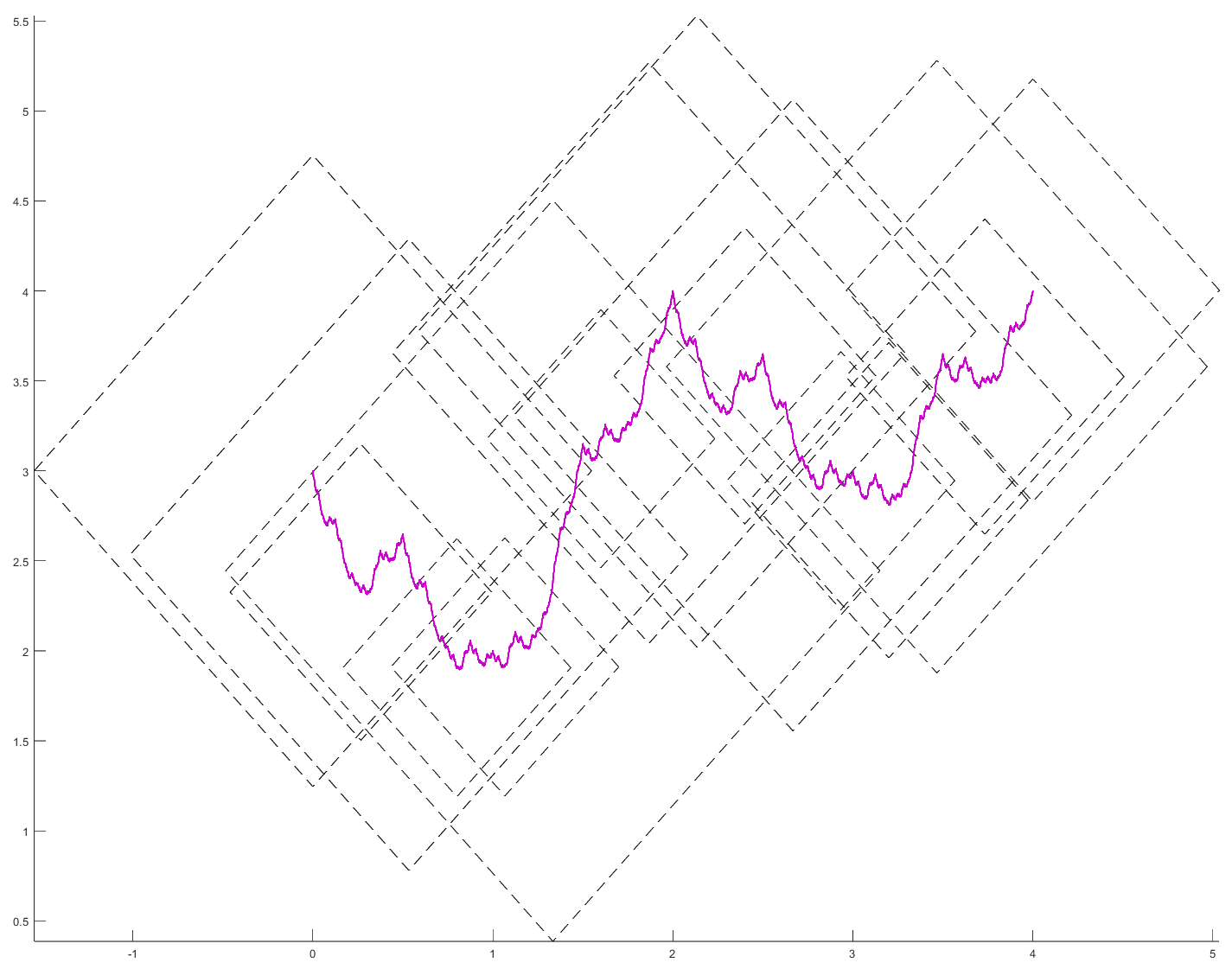}}
  \end{tabular}
\end{figure}

\begin{figure}[H]
  \centering
  \begin{tabular}{cc}
    \subfloat[\(F_{13}:\) Visualization of \(\mathcal{C}_3\)]{\includegraphics[width = 0.5\textwidth]{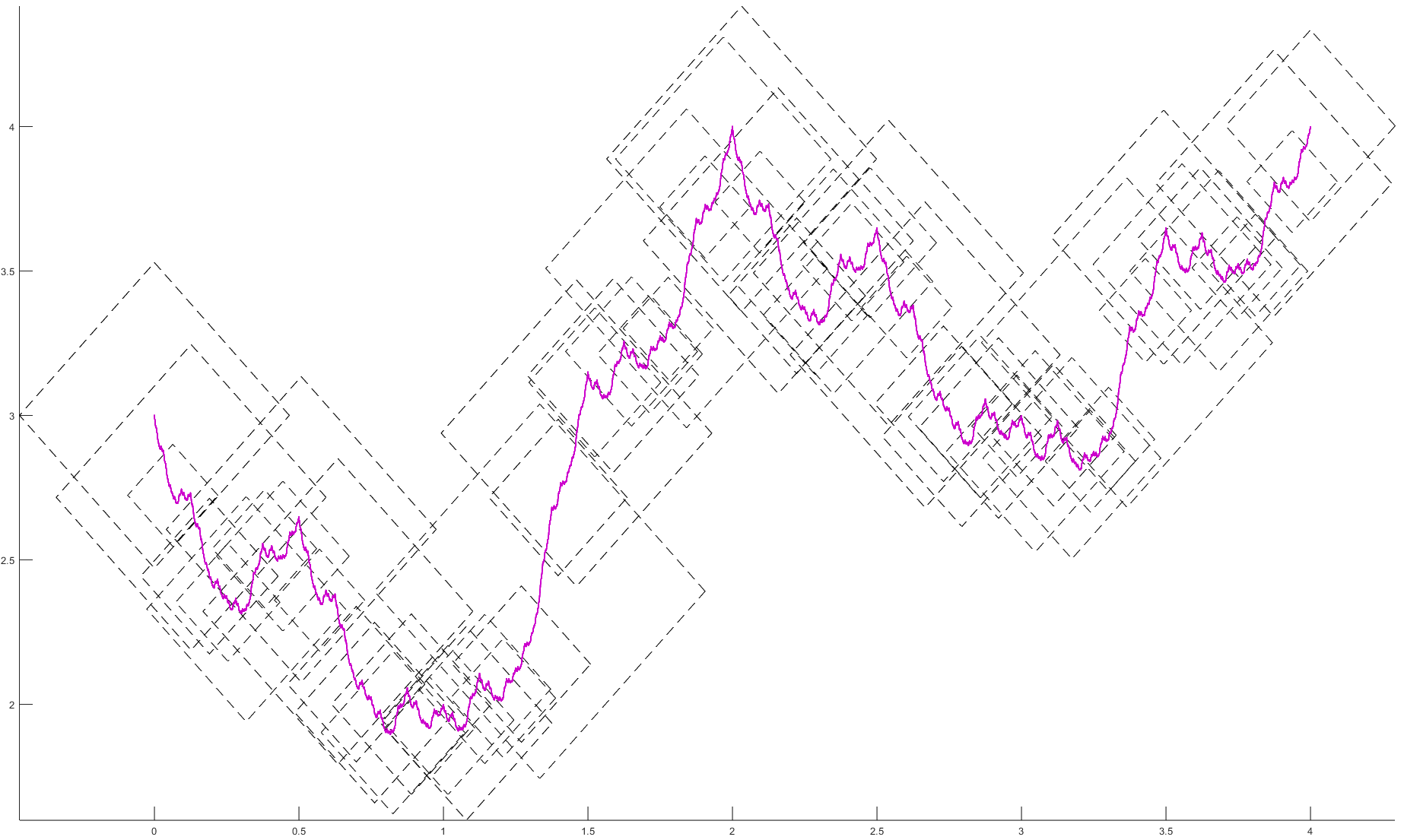}} &
    \subfloat[\(F_{14}:\) Visualization of \(\mathcal{C}_4\)]{\includegraphics[width = 0.5\textwidth]{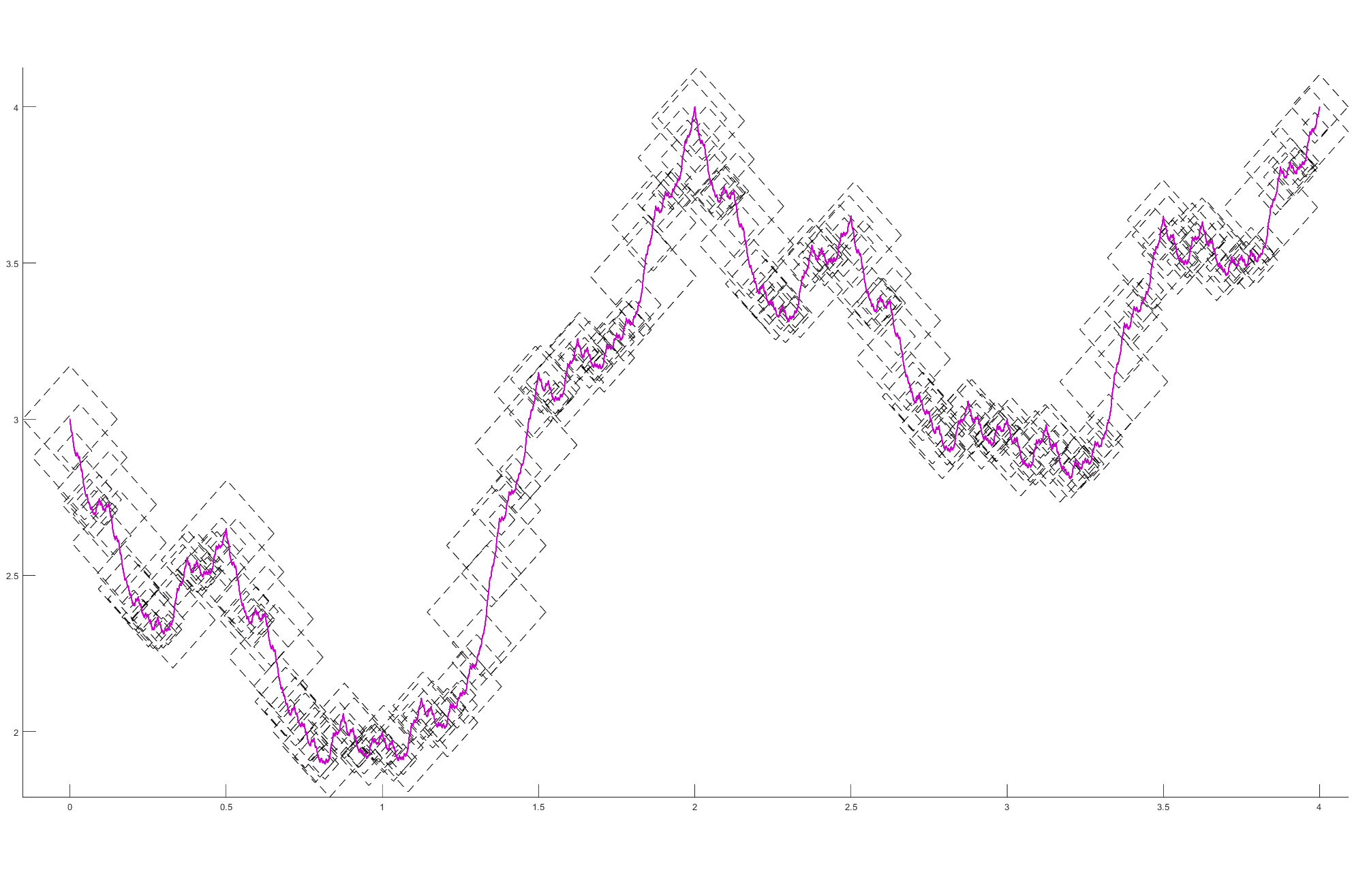}}
  \end{tabular}
\end{figure}

\begin{figure}[H]
  \centering
  \begin{tabular}{cc}
    \subfloat[\(F_{15}:\) Visualization of \(\mathcal{C}_5\)]{\includegraphics[width = 0.5\textwidth]{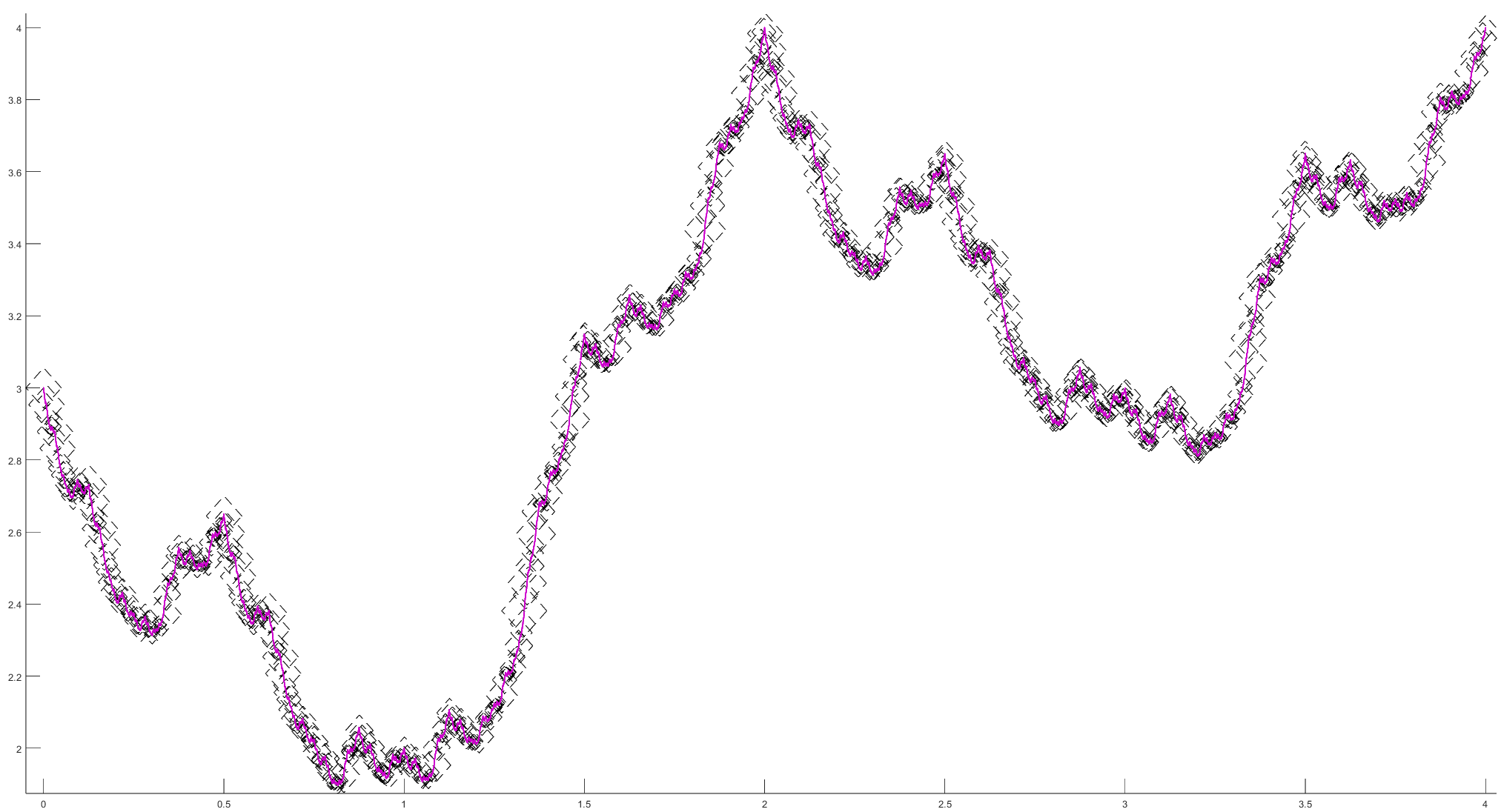}} &
  \end{tabular}
\end{figure}

Applying Theorem \ref{infinclusion} for \(\mathcal{S}_i\), \(i\in\{1,\dots, 5\}\), we get
\[Imf\subseteq [A_i,B_i],\] where:
\begin{figure}[H]
  \centering
  \begin{tabular}{c|c|c}
            & \(A_i\) & \(B_i\) \\
    \hline
    \(i=1\) & -5.2218 & 10.8386  \\
    \hline
    \(i=2\) & 0.3859 & 5.5299  \\
    \hline
    \(i=3\) & 1.5986  & 4.4169  \\
    \hline
    \(i=4\) & 1.7899  & 4.1260  \\
    \hline
    \(i=5\) & 1.8744  & 4.0393
  \end{tabular}.
\end{figure}

\textbf{Framework 2}\begin{itemize}\item the data set is
        \begin{center}
          \(\{(0,4),(1,2),(2,1),(3,5),(4,7),(5,4),(6,5),(7,2),(8,4),(9,5)\}\)
        \end{center}
  \item \;
        \begin{center}
          \(d_k=\frac{1}{4},\)
        \end{center}
\end{itemize}
for each \(k\in\{1,\dots,9\}\).

The pictures \(F_{21},F_{22},F_{23},F_{24}\) and \(F_{25}\) contain the graphical representations of \(\mathcal{C}_1\), \(\mathcal{C}_2\), \(\mathcal{C}_3\), \(\mathcal{C}_4\) and respectively \(\mathcal{C}_5\).
\captionsetup[subfigure]{labelformat=empty}
\begin{figure}[H]
  \centering
  \begin{tabular}{cc}
    \subfloat[\(F_{21}:\) Visualization of \(\mathcal{C}_1\)]{\includegraphics[width = 0.5\textwidth]{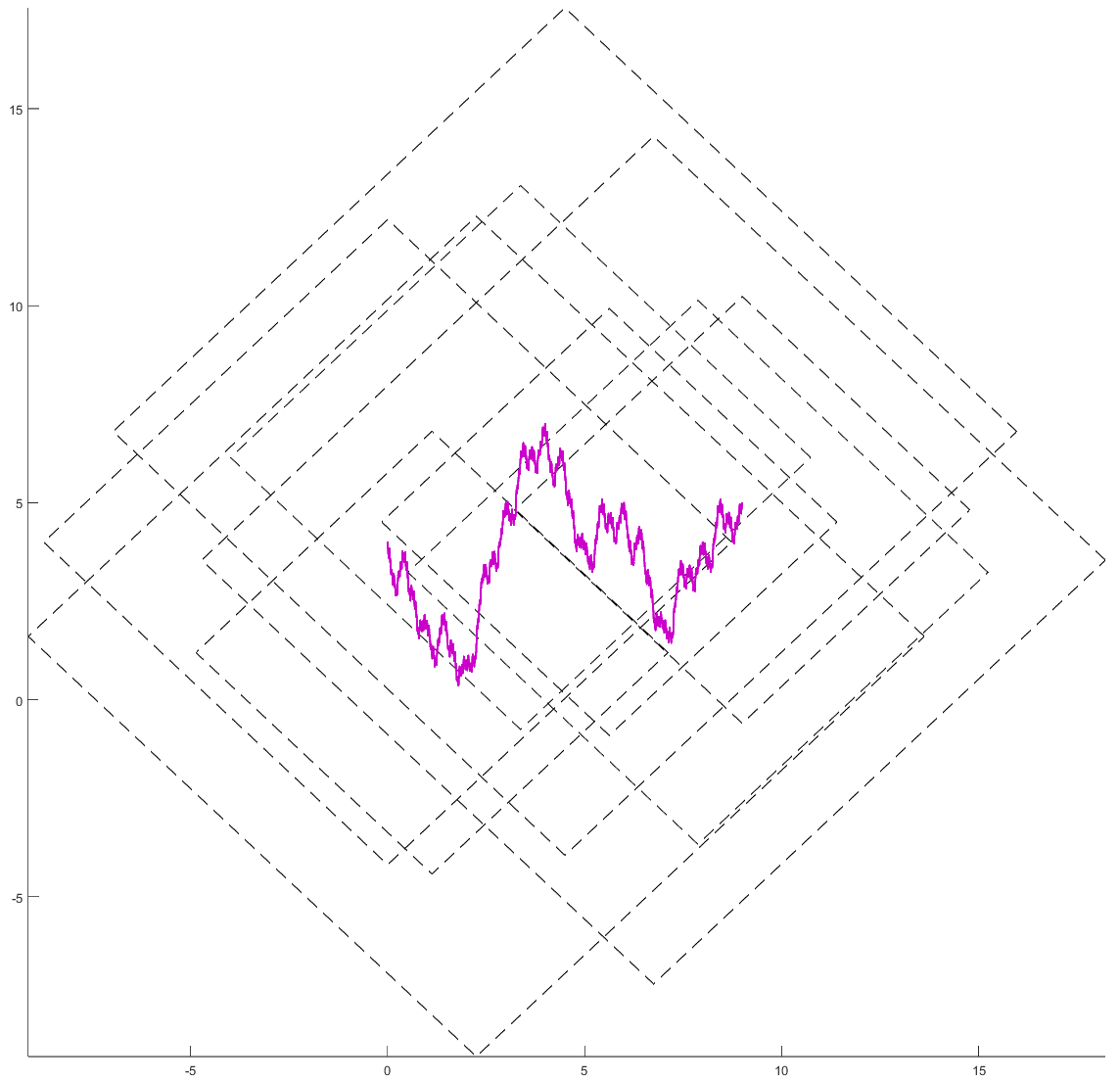}} &
    \subfloat[\(F_{22}:\) Visualization of \(\mathcal{C}_2\)]{\includegraphics[width = 0.5\textwidth]{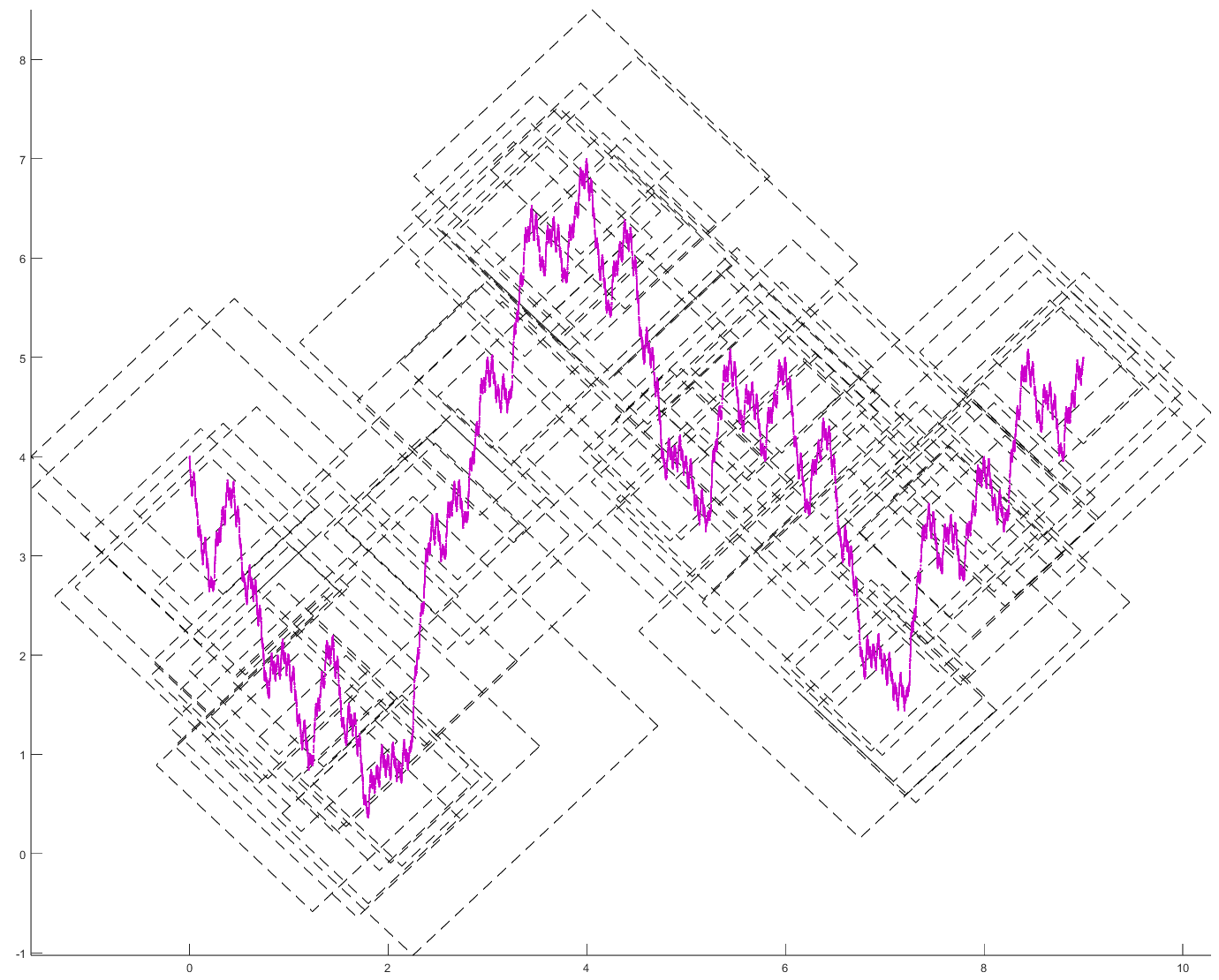}}
  \end{tabular}
\end{figure}

\begin{figure}[H]
  \centering
  \begin{tabular}{cc}
    \subfloat[\(F_{23}:\) Visualization of \(\mathcal{C}_3\)]{\includegraphics[width = 0.5\textwidth]{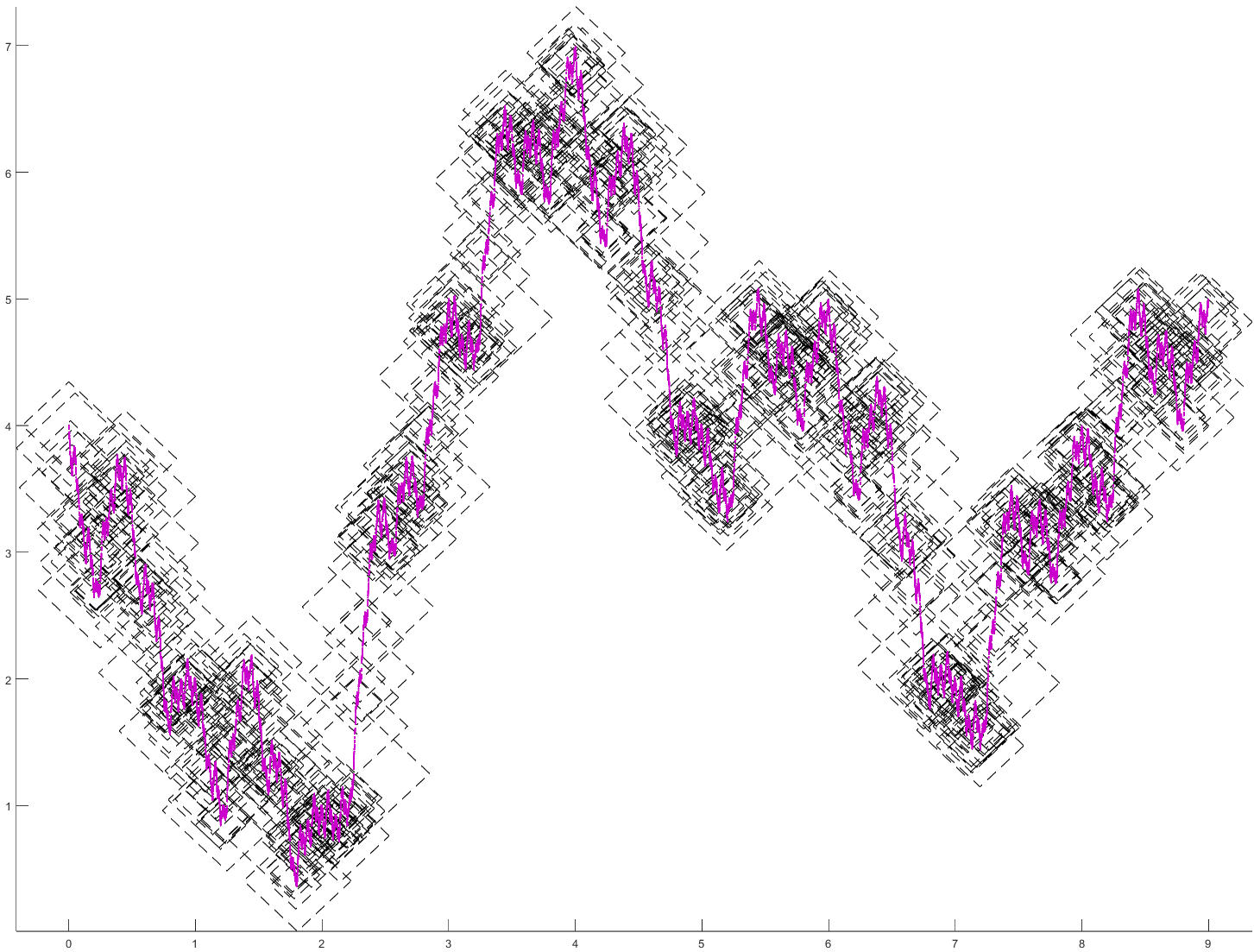}} &
    \subfloat[\(F_{24}:\) Visualization of \(\mathcal{C}_4\)]{\includegraphics[width = 0.5\textwidth]{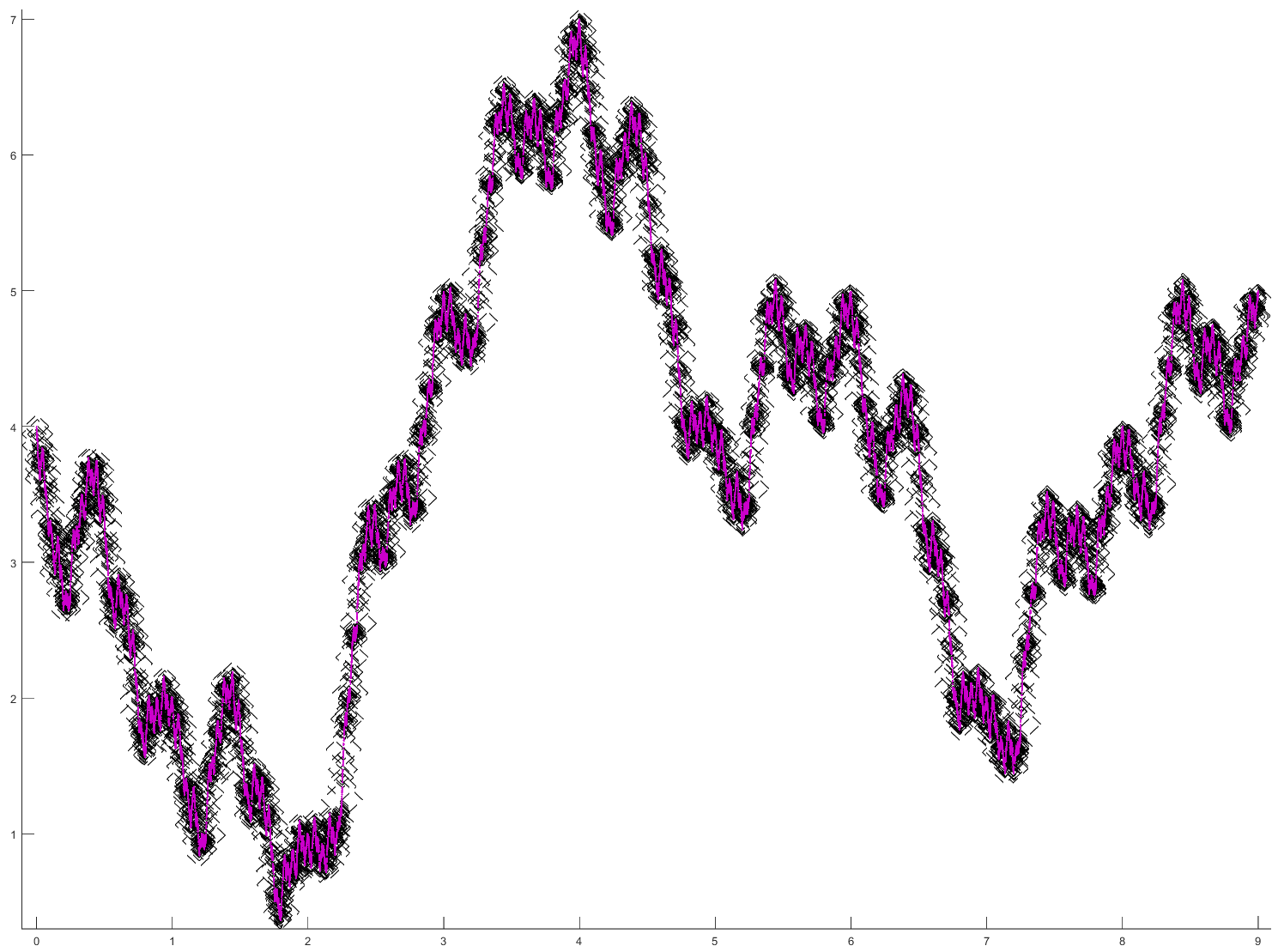}}
  \end{tabular}
\end{figure}

\begin{figure}[H]
  \centering
  \begin{tabular}{cc}
    \subfloat[\(F_{25}:\) Visualization of \(\mathcal{C}_5\)]{\includegraphics[width = 0.5\textwidth]{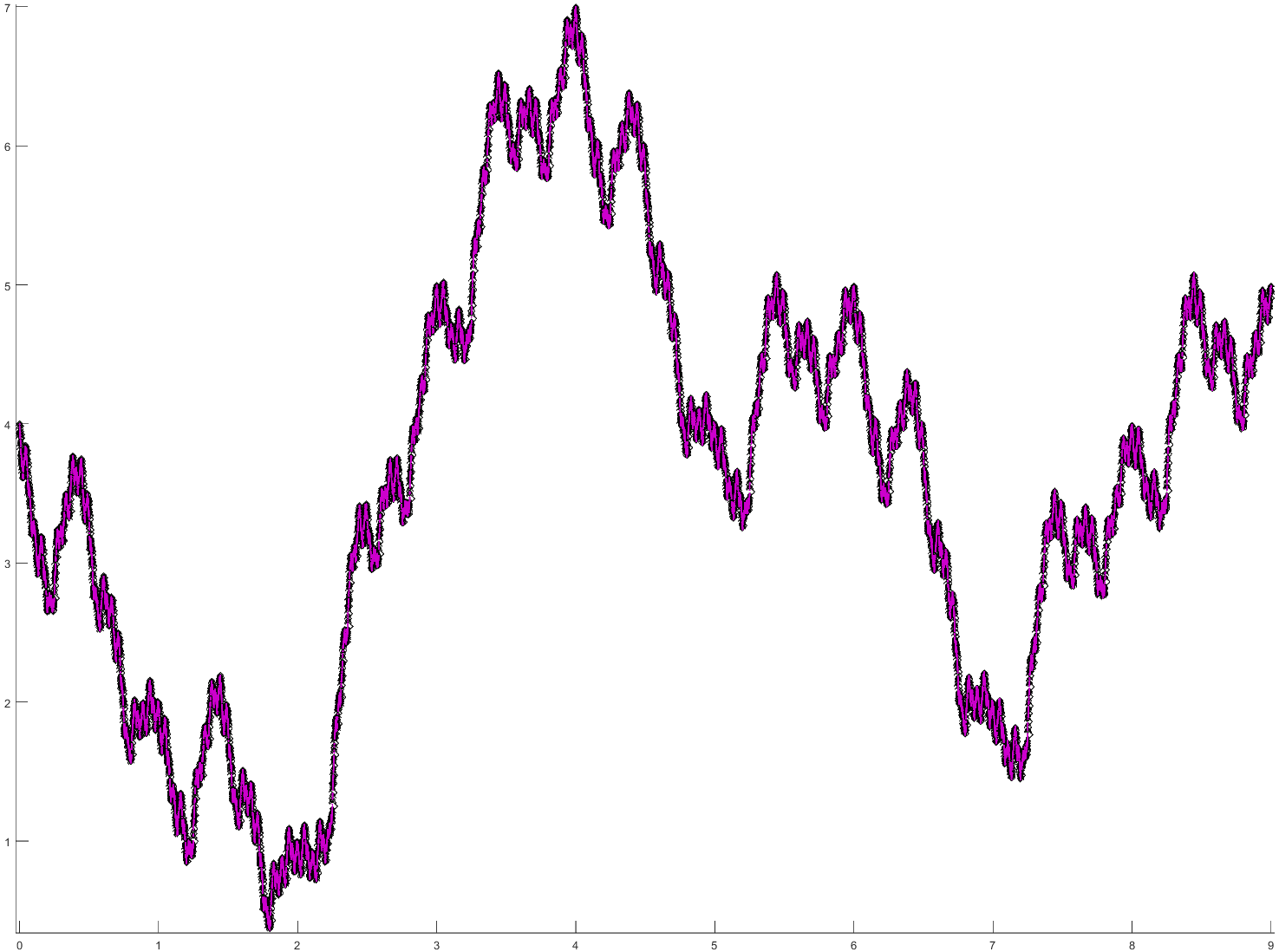}} &
  \end{tabular}
\end{figure}

Applying Theorem \ref{infinclusion} for \(\mathcal{S}_i\), \(i\in\{1,\dots, 5\}\), we get
\[Imf\subseteq [A_i,B_i],\] where
\begin{figure}[H]
  \centering
  \begin{tabular}{c|c|c}
            & \(A_i\) & \(B_i\) \\
    \hline
    \(i=1\) & -9.0538 & 17.5588
    \\
    \hline
    \(i=2\) & -1.0068 & 8.4890
    \\
    \hline
    \(i=3\) & 0.0069  & 7.3055
    \\
    \hline
    \(i=4\) & 0.2998  & 7.0716
    \\
    \hline
    \(i=5\) & 0.3393  & 7.0176
  \end{tabular}.
\end{figure}
\textbf{Framework 3}\begin{itemize}\item the data set is
        \begin{center}
          \(\{(0,0),(30,50),(60,40),(100,-10)\}\)
        \end{center}
  \item \;
        \begin{center}
          \(d_1=0.5,\;d_2=0.5\text{ and } d_3= 0.23.\)
        \end{center}
\end{itemize}
The pictures \(F_{31},F_{32},F_{33},F_{34}\) and \(F_{35}\) contain the graphical representations of \(\mathcal{C}_1\), \(\mathcal{C}_2\), \(\mathcal{C}_3\), \(\mathcal{C}_4\) and respectively \(\mathcal{C}_5\).
\captionsetup[subfigure]{labelformat=empty}
\begin{figure}[H]
  \centering
  \begin{tabular}{cc}
    \subfloat[\(F_{31}:\) Visualization of \(\mathcal{C}_1\)]{\includegraphics[width = 0.36\textwidth]{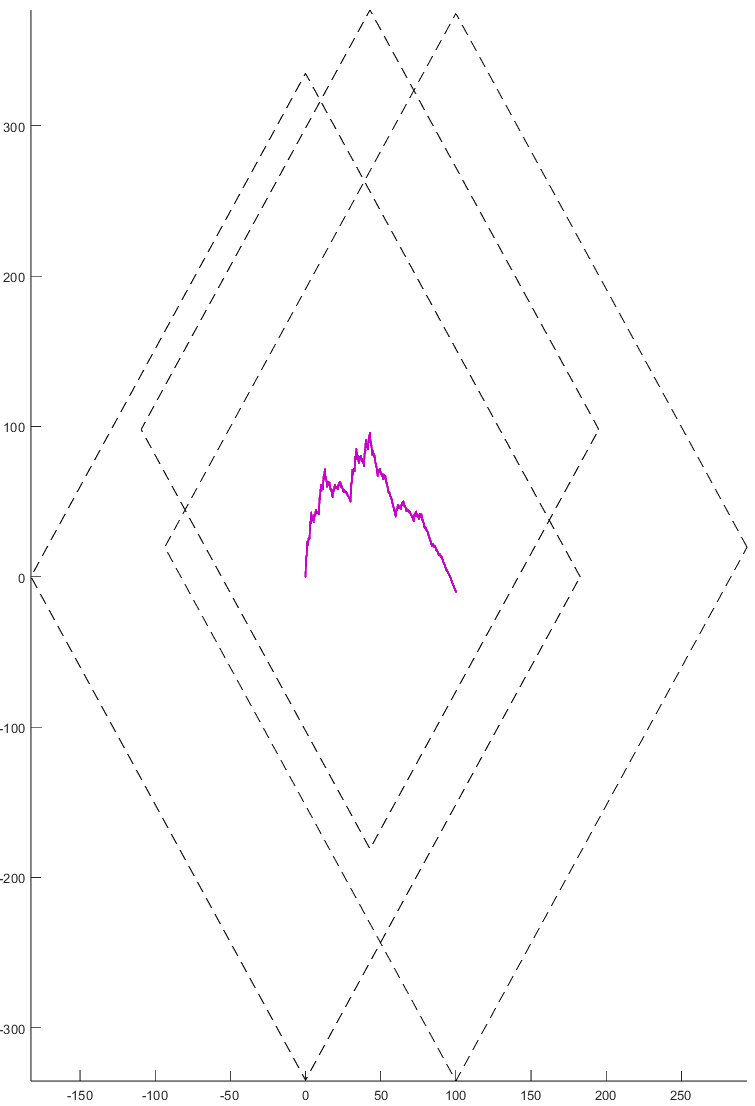}} &
    \subfloat[\(F_{32}:\) Visualization of \(\mathcal{C}_2\)]{\includegraphics[width = 0.36\textwidth]{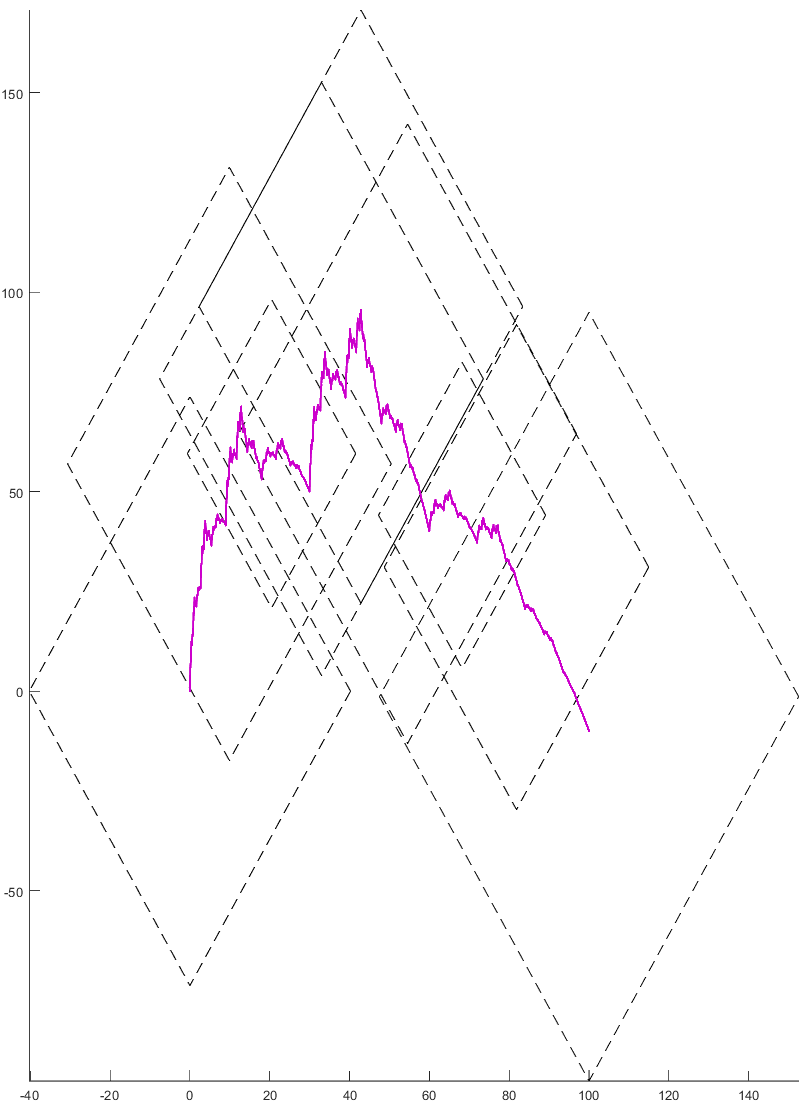}}
  \end{tabular}
\end{figure}

\begin{figure}[H]
  \centering
  \begin{tabular}{cc}
    \subfloat[\(F_{33}:\) Visualization of \(\mathcal{C}_3\)]{\includegraphics[width = 0.36\textwidth]{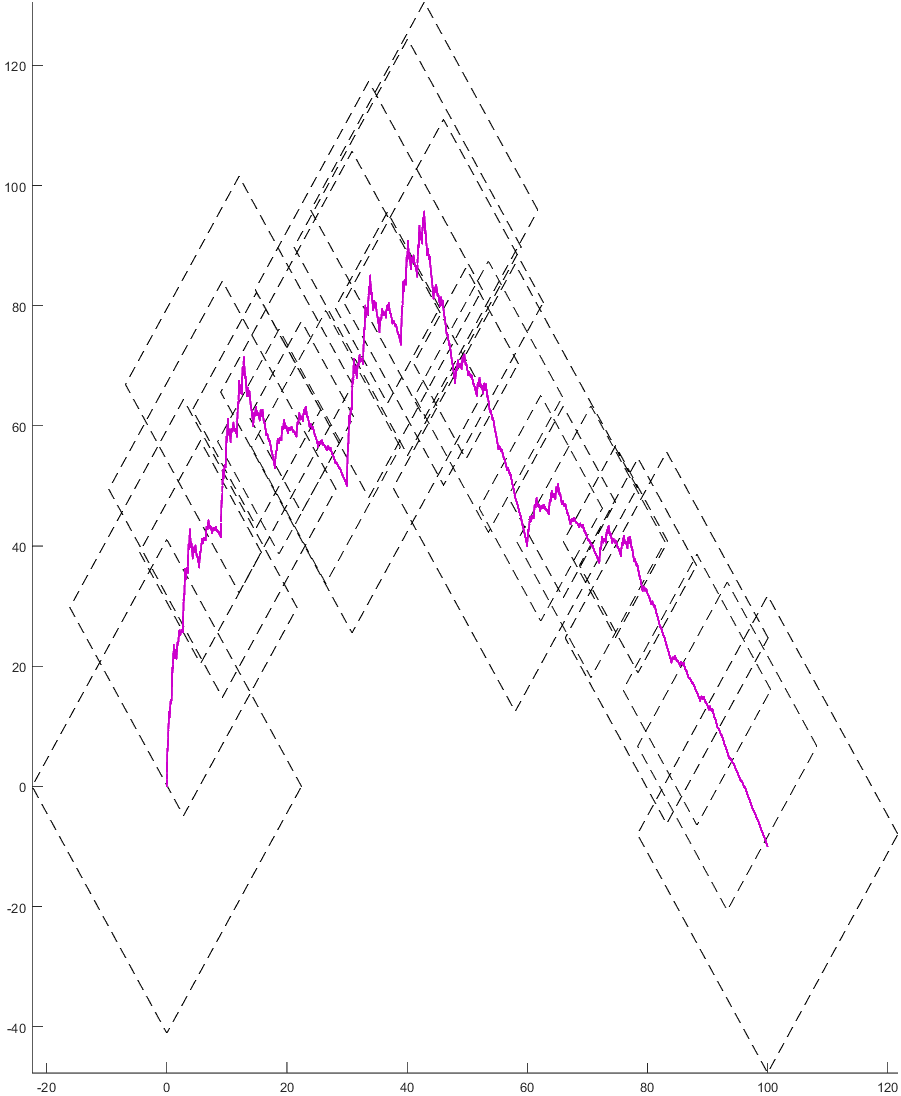}} &
    \subfloat[\(F_{34}:\) Visualization of \(\mathcal{C}_4\)]{\includegraphics[width = 0.36\textwidth]{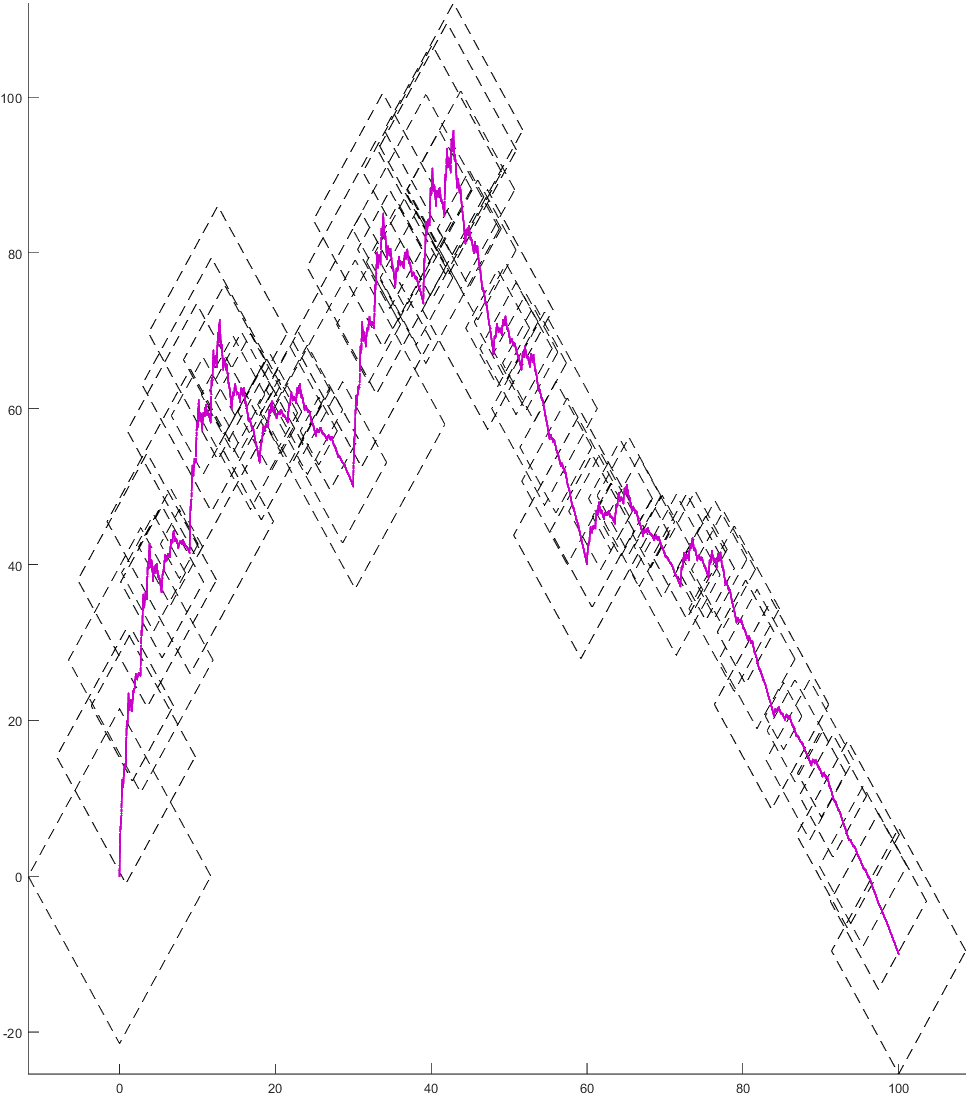}}
  \end{tabular}
\end{figure}

\begin{figure}[H]
  \centering
  \begin{tabular}{cc}
    \subfloat[\(F_{35}:\) Visualization of \(\mathcal{C}_5\)]{\includegraphics[width = 0.36\textwidth]{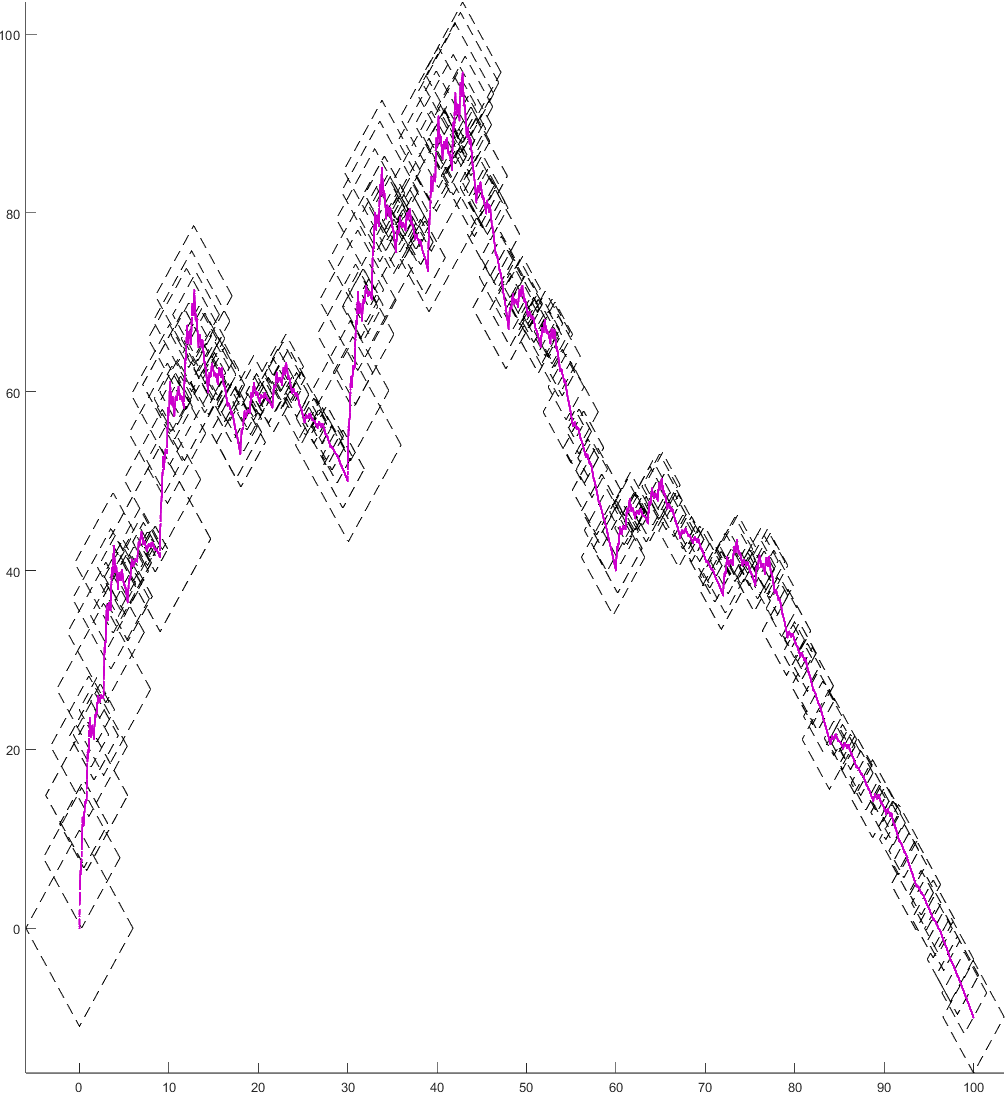}} &
  \end{tabular}
\end{figure}

Applying Theorem \ref{infinclusion} for \(\mathcal{S}_i\), \(i\in\{1,\dots, 5\}\), we get
\[Imf\subseteq [A_i,B_i],\] where
\begin{figure}[H]
  \centering
  \begin{tabular}{c|c|c}
            & \(A_i\)    & \(B_i\)  \\
    \hline
    \(i=1\) & -335.3988 &  376.9781
    \\
    \hline
    \(i=2\) & -97.6725  & 170.6729
    \\
    \hline
    \(i=3\) & -47.7096   & 130.5433
    \\
    \hline
    \(i=4\) & -25.3882  & 112.0746
    \\
    \hline
    \(i=5\) & -16.1975   & 103.5866
  \end{tabular}.
\end{figure}
\section*{Appendix}
\hyphenation{MATLAB}
In order to find the radiuses and vertices in the previous examples we used MATLAB R2021b and the following functions:

\begin{lstlisting}[style=Matlab-editor]
  %x and y are arrays of length n containing the data points.
  %d is an array consisting of the values d_k. 
  function [a,b,c,e,theta] = init_coeff(x,y,d)
  %This function is used to calculate the coefficients a_k,b_k,c_k,e_k and the constant theta.
  l=length(x);
  a=zeros(1,l-1);
  b=zeros(1,l-1);
  c=zeros(1,l-1);
  e=zeros(1,l-1);
  lip=zeros(1,l-1);
  fix =zeros(2,l-1);
  for i=1:l-1
      a(i)=(x(i+1)-x(i))/(x(l)-x(1));
      b(i)=(x(l)*x(i)-x(1)*x(i+1))/(x(l)-x(1));
      c(i)=(y(i+1)-y(i))/(x(l)-x(1))-d(i)*(y(l)-y(1))/(x(l)-x(1));
      e(i)=(x(l)*y(i)-x(1)*y(i+1))/(x(l)-x(1))-d(i)*(x(l)*y(1)-x(1)*y(l))/(x(l)-x(1));
  end

  if any(c)
      theta = (1-max(a))/(2*max(abs(c)));
  else
      theta=1;
  end

  end
\end{lstlisting}

\begin{lstlisting}[style=Matlab-editor]
  % m - order of composition
  % p - a list consisting of all m length arrays with numbers from the set {1,2, ... , n}.
  % k - an array of length m with numbers from the set {1,2, ... , n}.
  function [a_k,b_k,c_k,d_k,e_k] = comp_k(a,b,c,d,e,m,k)
    if m=1
      a_k=a(k);
      b_k=b(k);
      c_k=c(k);
      d_k=d(k);
      e_k=e(k);
    else
      [a_t, b_t, c_t, d_t, e_t] = comp_k(a,b,c,d,e,m-1,k(2:end));
      a_k = a(k(1))*a_t;
      b_k = a(k(1))*b_t;
      c_k = c(k(1))*a_t+d(k(1))*c_t;
      d_k = d(k(1))*d_t;
      e_t = c(k(1))*b_t+d(k(1))*e_t+e(k(1));
    end
  end

  function [lipschitz_constants, fixed_points] = comp_prop(a,b,c,d,e,m,p)
  %This function calculates the Lipschitz constants and fixed points of the system.
  l=length(p);
  lip=zeros(1,l);
  fix =zeros(2,l);
  for i = 1:p
    [a_m,b_m,c_m,d_m,e_m]=comp_k(a,b,c,d,e,m,p(i));
    lip(i)=max(d_m, a_m+theta*abs(c_m));
    fix(:,i)=[b_m/(1-a_m);b_m*c_m/((1-a_m)*(1-d_m))+e_m/(1-d_m)];
  end
  end
\end{lstlisting}
\begin{lstlisting}[style=Matlab-editor]
  % lcs - sorted array of Lipschitz constants calculated previously.
  % fps - array of the fixed points calculated previously, sorted in in the order given by the Lipschitz constants.
  function [r,v] = radiuses(lcs,fps)
  % this function calculates the radiuses and vertices for each function in the system.
  l=length(lcs);
  r=zeros(1,l);
  v=cell(1,l);
  maxd=0;
  for i=1:l-1
      for j = i+1:l
          n=norm(fps(:,i)-fps(:,j));
          if n>maxd
              maxd=n;
          end
      end
  end
  for i=1:l-1 
    r(i)=lcs(i)*maxd*(1+lcs(i))/(1-lcs(end)*lcs(end-1));
    v{i}=[[fps(1,i)+r(i);fps(2,i)],
          [fps(1,i)-r(i);fps(2,i)],
          [fps(1,i);fps(2,i)+r(i)/theta],
          [fps(1,i);fps(2,i)-r(i)/theta]];
    end
    r(l)=lcs(l)*maxd*(1+lcs(end-1))/(1-lcs(end)*lcs(end-1));
    v{l}=[[fps(1,l)+r(l);fps(2,l)],
          [fps(1,l)-r(l);fps(2,l)],
          [fps(1,l);fps(2,l)+r(l)/theta],
          [fps(1,l);fps(2,l)-r(l)/theta]];
  end
end

\end{lstlisting}

\makeatletter
\def\verbatim@font{\normalfont}
\makeatother
\begin{verbatim}
Bogdan-Cristian Anghelina
Faculty of Mathematics and Computer Science
Transilvania University of Brașov
Iuliu Maniu Street, nr. 50, 500091, Brașov, Romania
E-mail: bogdan.anghelina@unitbv.ro

Radu Miculescu
Faculty of Mathematics and Computer Science
Transilvania University of Brașov
Iuliu Maniu Street, nr. 50, 500091, Brașov, Romania
E-mail: radu.miculescu@unitbv.ro

María Antonia Navascués
Departamento de Matemática Aplicada
Universidad de Zaragoza
50018 Zaragoza Spain 
E-mail: manavas@unizar.es
\end{verbatim}
\end{document}